\documentclass[12pt]{article}
\usepackage{amsmath,amssymb,amsthm}
\usepackage{geometry}
\usepackage{color}
\usepackage[english]{babel}
\usepackage{enumitem}

\newtheorem{theorem}{Theorem}
\newtheorem{lemma}[theorem]{Lemma}

\theoremstyle{definition}

\newtheorem{remark}[theorem]{Remark}

\numberwithin{theorem}{section}
\numberwithin{equation}{section}

\usepackage{hyperref}

\newcommand{\R}{\mathbb{R}}

\title{Positive normalized solutions to a singular elliptic equation with a $L^2$-supercritical nonlinearity}
\author{ Name}
\date{}
\author{
\hspace{2mm}\hspace{2mm}\\
{\scshape Siyu Chen}\footnote{The author was partially supported by the NSFC (Grant No. 12501142), ZJNSF (Grant No. LQN25A010017), and CSC (Grant No. 202409610002),} \\
{\it\small College of Data Science, Jiaxing University}\\
{\it\small 314001, Jiaxing, Zhejiang, People's Republic of China}\\
{\it\small e-mail:  sychen@zjxu.edu.cn}
\vspace{1mm}
\hspace{2mm}\hspace{2mm}\\
{\scshape Xiaojun Chang}\footnote{The corresponding author was partially supported by the NSFC (Grant No. 12471102), JLNSF (Grant No. 20250102004JC),
Research Project of the Education Department of Jilin Province (Grant No. JJKH20250296KJ)}\\
{\it\small School of Mathematics and Statistics, Northeast Normal
University}\\
{\it\small 130024 Changchun, Jilin, People's Republic of China}\\
{\it\small E-mail: changxj100@nenu.edu.cn}
\vspace{1mm}
\hspace{2mm}\hspace{2mm}\\
{ \scshape Jiazheng Zhou}\footnote{The author was partially supported by CNPq/Brazil (Grant No. 304627/2023-2) and FAP-DF/Brazil (Grant No. 00193-00002209/2023-56)} \\
{\it\small Departamento de Matem\'atica, Universidade de Bras\'ilia}\\
{\it\small 70910-900 Bras\'ilia, DF, Brazil}\\
{\it\small e-mail: zhou@mat.unb.br
}\vspace{1mm}
\vspace{1mm}
}
\begin{document}

\maketitle

\begin{abstract}
This paper studies the existence of positive normalized solutions to the singular elliptic equation
\[
-\Delta u + \lambda u = u^{-r} + u^{p-1} \quad \text{in } \Omega,
\]
with the Dirichlet boundary condition $u=0$ on $\partial\Omega$ and the normalization constraint $\int_\Omega u^2\,dx = \rho$.
Here $\Omega\subset\R^N$ ($N\ge3$) is a smooth bounded domain, $0<r<1$, $2+\frac{4}{N}<p<2^*$, where $2^*$ is the critical Sobolev exponent, and $\lambda\in\R$ is a Lagrange multiplier.
We obtain that for sufficiently small $\rho>0$, the problem admits a positive solution $(\lambda,u)\in\R\times H_0^1(\Omega)$.
The proof is based on a variational approach using a regularized functional and a careful analysis of the limiting process.
\end{abstract}

\newcommand{\msc}[1]{%
  \noindent\textbf{2020 Mathematics Subject Classification:} #1
}
\newcommand{\keywords}[1]{%
  \begin{flushleft}
  \textbf{Key words:} #1
  \end{flushleft}
}

\msc{35B09, 35J20, 35J60, 35J75.}


\section{Introduction and main result}\label{sec:introduction}
In this paper, we study the existence of positive solutions to the following singular boundary value problem with a prescribed $L^2$-norm:

\begin{equation}\label{eq:main}
\begin{cases}
-\Delta u + \lambda u = u^{-r} + u^{p-1} & \text{in } \Omega, \\[2pt]
u = 0 & \text{on } \partial\Omega, \\[2pt]
\displaystyle \int_\Omega |u|^2 \, \mathrm{d}x = \rho,
\end{cases}
\end{equation}
where $\Omega \subset \R^N$ ($N \ge 3$) is a smooth bounded domain, $0 < r < 1$, $2_*:=2 + \frac{4}{N} < p < 2^* := \frac{2N}{N-2}$ (the critical Sobolev exponent), $\rho > 0$ is a prescribed mass, and $\lambda\in\R$ is a Lagrange multiplier.

Problem \eqref{eq:main} is motivated in particular by the study of standing wave solutions of form $\Phi(t,x)=e^{i\lambda t}u(x)$ to the following time-dependent nonlinear Schr\"odinger equation (NLS)
\begin{equation}\label{nonlinear}
\begin{aligned}
i\frac{\partial\Phi}{\partial t}+\Delta\Phi+f(|\Phi|)\Phi=0\ \ (t,x)\in\mathbb{R}\times\Omega. 
\end{aligned}
\end{equation}
The Schr\"odinger equation arises in many applications of mathematical physics, such as nonlinear optics and Bose-Einstein condensates. In this context, the domain $\Omega$ is either $\mathbb{R}^N$, or a bounded subset of $\mathbb{R}^N$(see \cite{Ag2013,BC2013,Cazenave2003,Fibich}).

A standard reduction leads to the stationary problem
\begin{equation}\label{NLS}
\begin{aligned}
-\Delta u + \lambda u &= g(u), \ \ u\in H_0^1(\Omega), 
\end{aligned}
\end{equation}
There are two main perspectives for studying \eqref{NLS}. The first, often called the fixed frequency problem, is to prescribe the parameter $\lambda\in\mathbb{R}$ and seek corresponding solutions of \eqref{NLS}\cite{Berestychi,Kwong}. The second approach takes into account a fundamental property of the time-dependent equation \eqref{nonlinear}: the conservation of mass
$$||\Psi(\cdot,t)||_{L^2(\Omega)}=||\Psi(\cdot,0)||_{L^2(\Omega)}.$$
It is therefore natural to prescribe the mass $||u||^2_{L^2(\Omega)}$ as a constraint and treat the frequency $\lambda$ as part of the unknown.
This leads to the so-called {\it normalized solutions} problem \cite{Lions1,Jean}, which is the focus of the present work.

Consider the functional $I: H_0^1(\Omega)\to\mathbb{R}$ given by
$$I(u)=\frac{1}{2}\int_\Omega|\nabla u|^2dx-\int_\Omega G(u)dx$$
on the constraint 
$$\mathcal{S}_\rho=\{u\in H_0^1(\Omega): \int_\Omega u^2dx=\rho\},$$
where $G(u)=\int_0^ug(s)ds$. Clearly, normalized solutions of \eqref{NLS} correspond to critical points of the functional $I$ on $\mathcal{S}_\rho$.

When $\Omega=\mathbb{R}^N$, much attention has been paid to the existence of normalized solutions, and many techniques have been developed recently.
In the mass-subcritical case, i.e., when $g(u)\sim|u|^{p-2}u$ as $|u|\to +\infty$, with $2<p<2_*$, the functional $I$ is bounded from below due to the Gagliardo-Nirenberg inequality. Hence one can obtain a global minimizer by minimizing $I$ on $\mathcal{S}_\rho$ (see \cite{Lions,Stuart,Shibata-2014}). As for the mass-supercritical case, i.e., $g(u)\sim|u|^{p-2}u$ as $|u|\to +\infty$, with $2_*<p<2^*$, the energy functional $I$ is unbounded both from above and below, so minimization arguments are not applicable. One is then forced to search for critical points that are not global minima. Using a mountain-pass argument on constraint for an auxiliary functional in the seminal paper \cite{Jean}, Jeanjean showed the existence of a normalized solution for \eqref{NLS} in $\mathbb{R}^N$. More problems involving normalized solutions in $\mathbb{R}^N$, for both a single equation and systems, have been studied under various assumptions; see for example \cite{AJM2021,BS2017-1,BZZ-2021,Mederski-2021,JL-2022,JL-2020,MS-2024,Soave-2020-1,Soave-2020-2,WW-2022} and the references therein.
%For more general nonlinearity $g$, the existence of normalized solutions becomes nontrivial and many techniques have been developed recently. 

By contrast, problem \eqref{NLS} on a bounded domain $\Omega$ is still far from being well understood. Owing to the absence of scaling invariance and the presence of uncontrollable boundary terms in the Pohozaev identities, the analysis becomes more intricate. For the pure-power nonlinearity $g(u)=|u|^{p-2}u$ in the $L^2$-supercritical and Sobolev subcritical regime, Noris et al. \cite{NTV-2014} demonstrated the existence of a positive normalized solution on a unit ball with Dirichlet boundary conditions. This was subsequently extended to general bounded domains in \cite{PV-2017}. The Sobolev critical case was studied in \cite{NTV-2019,PVY-2025}, and related results for nonlinear Schr\"odinger systems can be found in \cite{BQZ-2024,NTV-2015}. 
For more contributions to this topic, we refer to
 \cite{AS-2025,CRZ-2025,Dovetta-2025,PPVV-2021,SZ-2025} and their references.

Singular nonlinear problems were first introduced by Fulks and Maybee \cite{MR123095} as a mathematical model for describing the heat conduction in an electrical conductor and received considerable attention after the pioneering paper of Grandall, Rabinowitz and Tartar \cite{MR427826}. There is an extensive literature dealing with singular term of the type $u^{-\gamma}$ with $0<\gamma<1$. In such a case, one can associate the problem with an energy functional that, although not continuously $G\hat{a}teaux$ differentiable, is strictly convex.

Singular problems with linear or superlinear perturbations have also attracted a lot of interest. Consider
\begin{equation}\label{Perturbation}
\left\{
\begin{aligned}
-\Delta u &= \lambda u^{-\gamma} + u^{\beta} \quad \text{in } \Omega, \\
u &> 0 \quad \text{in } \Omega,~ u = 0 \quad \text{on } \partial\Omega.
\end{aligned}
\right.
\end{equation}
For $\beta \leq 1$ and $0<\gamma<1$, uniqueness was established by Brezis and Oswald in \cite{MR820658}. When $0<\gamma <1< \beta<2^*-1$, the solutions in $H_0^1(\Omega)$ are not unique; in fact, multiplicity results have been proved in \cite{ MR3058700}. For the critical case $\beta=2^*-1$, we refer the reader to \cite{MR2749428,MR1964476}.

Singular elliptic equations with power-type nonlinearities and mass constraints remain largely unexplored. The main difficulty in studying such problems comes from the interplay between the singular term $u^{-r}$ and the $L^2$-norm constraint. The singular term introduces a strong nonlinearity that becomes unbounded as $u \to 0^+$, making the problem non-Fr\'{e}chet differentiable at the origin. Meanwhile, the mass constraint imposes a nontrivial variational structure, requiring us to work on the sphere $\{u \in H_0^1(\Omega): u \ge 0,\ \|u\|_2^2 = \rho\}$.

Our main result is the existence of positive solutions to problem \eqref{eq:main}  when the mass $\rho$ is sufficiently small.

\begin{theorem}\label{thm:main}
Let $\Omega \subset \R^N$ ($N \ge 3$) be a smooth bounded domain, $0 < r < 1$, and $2_* < p < 2^*$. Then there exists $\rho_0 > 0$ such that for any $\rho \in (0, \rho_0)$, problem \eqref{eq:main} admits a solution $(\lambda, u) \in \R \times H_0^1(\Omega)$ with $u > 0$ in $\Omega$.
\end{theorem}

The proof employs a variational approach combined with a regularization technique. We introduce a family of regularized functionals
\[
J_\varepsilon(u) = \frac12 \int_\Omega |\nabla u|^2 \,dx - \frac{1}{1-r} \int_\Omega (u+\varepsilon)^{1-r} \,dx - \frac1p \int_\Omega u^p \,dx,
\]
where $\varepsilon > 0$ removes the singularity. For each $\varepsilon$, we minimize $J_\varepsilon$ on a suitable closed subset of the constraint set (defined in Section~\ref{sec:preliminaries}) and obtain a positive solution $(\lambda_\varepsilon,u_\varepsilon)$ of the regularized problem using the Lagrange multiplier rule. The heart of the argument  hinges on deriving uniform estimates (independent of $\varepsilon$) for both $\{\lambda_\varepsilon\}$ and $\{u_\varepsilon\}$, including a key uniform $L^\infty$ bound for $\{u_{\epsilon}\}$ obtained via a blow-up analysis. Equipped with these bounds, passing to the limit $\varepsilon \to 0^+$ then yields the desired solution of the original problem.

The paper is organized as follows. Section~\ref{sec:preliminaries} introduces the variational framework and proves the existence of minimizers for the regularized problems. In Section~\ref{sec:estimates} we derive uniform bounds for the Lagrange multipliers, and for the $L^\infty$ norm of the minimizers via a blow-up argument; then we pass to the limit and verify that the limit is a positive solution of the original problem, thus completing the proof of Theorem~\ref{thm:main}. Finally, some concluding remarks are provided in  Section~\ref{sec:conclusion}.

\section{Preliminaries and existence of minimizers}\label{sec:preliminaries}

\subsection{Notations and functional setting}
Let $\Omega \subset \R^N$ ($N\ge3$) be a smooth bounded domain.
We denote by $H_0^1(\Omega)$ the usual Sobolev space with norm $\|u\|_{H_0^1} = \left( \int_\Omega |\nabla u|^2\,dx \right)^{1/2}$.
The Sobolev embedding $H_0^1(\Omega) \hookrightarrow L^q(\Omega)$ is continuous for $1 \le q \le 2^*$ and compact for $1 \le q < 2^*$.
We will frequently use the Gagliardo-Nirenberg inequality: for $2 \le q \le 2^*$, there exists a constant $C>0$ such that
\[
\|u\|_q \le C \|\nabla u\|_2^{a} \|u\|_2^{1-a}, \quad \text{with } a = N\left( \frac12 - \frac1q \right).
\]

Define the constraint set
\[
S_{\rho} = \left\{ u \in H_0^1(\Omega): u \ge 0, \ \int_\Omega u^2\,dx = \rho \right\}.
\]
For $\varepsilon > 0$, consider the regularized energy functional $J_\varepsilon : H_0^1(\Omega) \to \R$ given by
\[
J_\varepsilon(u) = \frac12 \int_\Omega |\nabla u|^2\,dx - \frac{1}{1-r} \int_\Omega (u+\varepsilon)^{1-r}\,dx - \frac1p \int_\Omega u^p\,dx.
\]
Notice that for $u \in S_{\rho}$, the term $\int_\Omega (u+\varepsilon)^{1-r}\,dx$ is well defined because $0<r<1$ and $u\ge0$.

\subsection{Geometry of the regularized functional and existence of minimizers}

The first step is to understand the geometry of $J_\varepsilon$ on the constraint set $S_{\rho}$. We begin by showing that $J_\varepsilon$ is positive on a suitable sphere in $H_0^1(\Omega)$.

\begin{lemma}\label{lem:geometry}
There exist constants $\rho_0 > 0$ and $\tau > 0$ such that for any $\rho \in (0, \rho_0)$ and any $\varepsilon \in (0,1]$, the following holds:
For every $u \in S_{\rho}$ with $\|\nabla u\|_2 = \tau$, we have $J_\varepsilon(u) > 0$.
\end{lemma}

\begin{proof}
For $u \in S_{\rho}$, using the Gagliardo-Nirenberg inequality, there exists $C>0$ such that
\[
\|u\|_p \le C \|\nabla u\|_2^{a} \|u\|_2^{1-a}, \quad a = N\left( \frac12 - \frac1p \right) \in (0,1).
\]
Since $2+4/N < p < 2^*$ and $a = N(\frac12 - \frac1p)$, we have $ap = \frac{Np}{2} - N > 2$.
Since $\|u\|_2^2 = \rho$, we obtain
\[
\int_\Omega u^p \,dx \le C^p \rho^{(1-a)p/2} \|\nabla u\|_2^{ap} =: C_2 \|\nabla u\|_2^{ap},
\]
with $C_2 = C^p \rho^{(1-a)p/2}$.

For the singular term, it follows from $0<r<1$ and H\"older's inequality that
\[
\int_\Omega (u+\varepsilon)^{1-r} \,dx \le \int_\Omega u^{1-r} \,dx + \varepsilon^{1-r} |\Omega| \le \rho^{(1-r)/2} |\Omega|^{(1+r)/2} + |\Omega| =: C_1,
\]
where we used $\varepsilon \le 1$.
Thus,
\[
J_\varepsilon(u) \ge \frac12 \|\nabla u\|_2^2 - \frac{C_1}{1-r} - \frac{C_2}{p} \|\nabla u\|_2^{ap} =: g(\|\nabla u\|_2),
\]
with $g(t) = \frac12 t^2 - \frac{C_1}{1-r} - \frac{C_2}{p} t^{ap}$ for $t \ge 0$.
Clearly, $g'(t) = t - C_2 a t^{ap-1} = t(1 - C_2 a t^{ap-2})$.
For $t>0$, setting $g'(t)=0$ gives $1 - C_2 a t^{ap-2} = 0$, hence the unique positive critical point
\[
t_0 = \left( \frac{1}{C_2 a} \right)^{1/(ap-2)} > 0.
\]

Now, we analyze the sign of $g'(t)$. For $0 < t < t_0$, the inequality
\[
t^{ap-2} < t_0^{ap-2} = \frac{1}{C_2 a}
\]
implies $C_2 a t^{ap-2} < 1$. Consequently, 
\[
g'(t) = t(1 - C_2 a t^{ap-2}) > 0.
\]
For $t > t_0$, we have
\[
t^{ap-2} > t_0^{ap-2} = \frac{1}{C_2 a},
\]
hence $C_2 a t^{ap-2} > 1$, and therefore
\[
g'(t) = t(1 - C_2 a t^{ap-2}) < 0.
\]
Thus, $g(t)$ increases strictly on $(0, t_0)$ and decreases strictly on $(t_0, \infty)$.
Since $ap>2$, we have $\lim_{t\to0^+} g(t) = -\frac{C_1}{1-r} < 0$ and $\lim_{t\to\infty} g(t) = -\infty$.
Therefore, $t_0$ is the unique global maximum point of $g$ on $[0,\infty)$.

Evaluating $t_0$, we have
\[
  g(t_0) = \frac12 t_0^2 - \frac{C_1}{1-r} - \frac{C_2}{p} t_0^{ap}.
\]
From $t_0^{ap} = t_0^2 \cdot t_0^{ap-2} = t_0^2 \cdot \frac{1}{C_2 a}$, we obtain
\begin{equation}\label{eq:g_max}
g(t_0) = \left( \frac12 - \frac{1}{ap} \right) t_0^2 - \frac{C_1}{1-r}.
\end{equation}
Again, since $ap>2$, we have $\frac12 - \frac{1}{ap} > 0$.

Now, as $\rho \to 0$, we obtain $C_2 = C^p \rho^{(1-a)p/2} \to 0$ and
$C_1 = \rho^{(1-r)/2} |\Omega|^{(1+r)/2} + |\Omega| \to |\Omega|$.
Consequently, $t_0 = \left( \frac{1}{C_2 a} \right)^{1/(ap-2)} \to \infty$.
For the behavior of $g(t_0)$ as $\rho \to 0$. Recall \eqref{eq:g_max},
since $\frac12 - \frac{1}{ap} > 0$ and $t_0^2 \to \infty$, the first term tends to $+\infty$,
while the second term remains bounded because $C_1 \to |\Omega|$. Thus,
\[
\lim_{\rho\to 0} g(t_0) = +\infty.
\]
In particular, there exists $\rho_0 > 0$ such that for every $\rho \in (0, \rho_0)$, $g(t_0) > 0$.

Set $\tau = t_0$.
Then we have $J_\varepsilon(u) \ge g(\tau) > 0$ for any $u \in S_{\rho}$ with $\|\nabla u\|_2 = \tau$.
\end{proof}

\begin{remark} The function $g(t)$ provides only a lower bound for $J_\varepsilon(u)$ when $\|\nabla u\|_2 = t$.
The monotonicity of $g$ on $(0, t_0)$ does not imply any monotonicity of $J_\varepsilon$ itself; it merely ensures that on the sphere $\|\nabla u\|_2 = \tau := t_0$, the functional $J_\varepsilon$ is positive.
\end{remark}

For each $\varepsilon \in (0,1]$, we now consider the closed subset of $S_{\rho}$ defined by
\[
K = \{ u \in S_{\rho} : \|\nabla u\|_2 \le \tau \},
\]
where $\tau$ is given in Lemma~\ref{lem:geometry}. The next lemma shows that the infimum of $J_\varepsilon$ on $K$ is negative.

\begin{lemma}\label{lem:negative_energy}
For each $\varepsilon \in (0,1]$, we have $c_\varepsilon := \inf_{u \in K} J_\varepsilon(u) < 0$.
\end{lemma}

\begin{proof}
We construct a specific function in $K$ that has negative energy. Let $\varphi \in C_c^\infty(\Omega)$ be a nonnegative function with $\int_\Omega \varphi^2\,dx > 0$. %Such a function exists; for instance, take a smooth nonnegative bump function with compact support in $\Omega$.
Define 
\[
c = \sqrt{ \frac{\rho}{\int_\Omega \varphi^2 \,dx} },
\]
and set $u_0 = c \varphi$. Then $u_0 \ge 0$ and
\[
\int_\Omega u_0^2 \,dx = c^2 \int_\Omega \varphi^2 \,dx = \rho,
\]
that is, $u_0 \in S_{\rho}$.

Next, we show that $u_0\in K$, i.e., $\|\nabla u_0\|_2 \le \tau$. Since
\[
\|\nabla u_0\|_2 = c \|\nabla \varphi\|_2 = \sqrt{\frac{\rho}{\int_\Omega \varphi^2 \,dx}} \|\nabla \varphi\|_2,
\]
by choosing $\rho$ sufficiently small (we may further reduce $\rho_0$ if necessary such that $\rho<\rho_0$), we can make $c$ arbitrarily small. Moreover, we may select $\varphi$ such that $\|\nabla \varphi\|_2$ is small. Thus, for sufficiently small $\rho$, we have $\|\nabla u_0\|_2 \le \tau$. Thus, $u_0 \in K$.

Now we compute $J_\varepsilon(u_0)$:
\[
J_\varepsilon(u_0) = \frac12 \int_\Omega |\nabla u_0|^2 \,dx - \frac{1}{1-r} \int_\Omega (u_0+\varepsilon)^{1-r} \,dx - \frac{1}{p} \int_\Omega u_0^p \,dx.
\]
Substituting $u_0 = c\varphi$, we obtain
\[
J_\varepsilon(u_0) = \frac12 c^2 \|\nabla \varphi\|_2^2 - \frac{1}{1-r} \int_\Omega (c\varphi + \varepsilon)^{1-r} \,dx - \frac{1}{p} c^p \|\varphi\|_p^p.
\]

We now estimate the singular term from below independently of $\varepsilon$. Since $c\varphi \ge 0$ and $\varepsilon \ge 0$, we have $(c\varphi + \varepsilon)^{1-r} \ge (c\varphi)^{1-r}$. Therefore,
\[
\int_\Omega (c\varphi + \varepsilon)^{1-r} \,dx \ge \int_\Omega (c\varphi)^{1-r} \,dx = c^{1-r} \int_\Omega \varphi^{1-r} \,dx.
\]
Note that $\int_\Omega \varphi^{1-r} \,dx > 0$ because $\varphi$ is nonnegative, not identically zero, and $1-r > 0$.
Thus,
\[
J_\varepsilon(u_0) \le \frac12 c^2 \|\nabla \varphi\|_2^2 - \frac{1}{1-r} c^{1-r} \int_\Omega \varphi^{1-r} \,dx.
\]
Since $0 < 1-r < 1$, the term $c^{1-r}$ decays more slowly than $c^2$ as $c \to 0$. Consequently, we can choose $c$ (and hence $\rho$) sufficiently small, independent of $\varepsilon$, such that
\[
\frac12 c^2 \|\nabla \varphi\|_2^2 < \frac{1}{1-r} c^{1-r} \int_\Omega \varphi^{1-r} \,dx.
\]
For such a choice, we have $J_\varepsilon(u_0) < 0$.

Since $u_0 \in K$, we conclude that
\[
c_\varepsilon = \inf_{u \in K} J_\varepsilon(u) \le J_\varepsilon(u_0) < 0.
\]
This completes the proof.
\end{proof}

With these geometric properties established, we first construct a minimizing sequence. Applying a compactness argument to this sequence then yields a minimizer.

\begin{lemma}\label{lem:minimizing_sequence}
For each $\varepsilon \in (0,1]$, there exists a minimizing sequence $\{u_n\} \subset K$ such that
\[
J_\varepsilon(u_n) \to c_\varepsilon, \quad \|\nabla u_n\|_2 < \tau,\ \forall n,
\]
and $\{u_n\}$ is bounded in $H_0^1(\Omega)$.
\end{lemma}

\begin{proof}
It follows from the definition of $c_{\varepsilon}$ and Lemma~\ref{lem:negative_energy}, that there exists a sequence $\{u_n\} \subset K$ such that
\[
J_\varepsilon(u_n) \to c_\varepsilon < 0.
\]
Since $c_\varepsilon < 0$, we may assume without loss of generality that $J_\varepsilon(u_n) < 0$ for all $n$. Now, if $\|\nabla u_n\|_2 = \tau$ for some $n$, then by Lemma~\ref{lem:geometry} we would have $J_\varepsilon(u_n) > 0$, a contradiction. Hence, $\|\nabla u_n\|_2 < \tau$ for all $n$. So $\{u_n\}$ is bounded in $H_0^1(\Omega)$. This completes the proof.
\end{proof}

\begin{lemma}\label{lem:minimizer}
For each $\varepsilon \in (0,1]$, there exists $u_\varepsilon \in K$ such that
\[
J_\varepsilon(u_\varepsilon) = c_\varepsilon, \quad \|\nabla u_\varepsilon\|_2 < \tau,
\]
and $u_\varepsilon > 0$ in $\Omega$.
\end{lemma}

\begin{proof}
Let $\{u_n\}$ be the minimizing sequence from Lemma~\ref{lem:minimizing_sequence}. 
Since $\{u_n\}$ is bounded in $H_0^1(\Omega)$, there exists a subsequence (still denoted by $\{u_n\}$) and $u_\varepsilon \in H_0^1(\Omega)$ such that $u_n \rightharpoonup u_\varepsilon$ weakly in $H_0^1(\Omega)$. 
By compact embedding $H_0^1(\Omega) \hookrightarrow L^q(\Omega)$ for $q < 2^*$, we have $u_n \to u_\varepsilon$ strongly in $L^q(\Omega)$ for all $q \in [1, 2^*)$. 
In particular, $u_n \to u_\varepsilon$ in $L^2(\Omega)$, so $\int_\Omega u_\varepsilon^2 \,dx = \rho$, and since $u_n \ge 0$, we have $u_\varepsilon \ge 0$ a.e. 
Thus $u_\varepsilon \in S_\rho$.
By weak lower semicontinuity of the norm, 
$\|\nabla u_\varepsilon\|_2 \le \liminf_{n\to\infty} \|\nabla u_n\|_2 \le \tau$, 
so $u_\varepsilon \in K$.

Now, we show that $J_\varepsilon(u_n) \to J_\varepsilon(u_\varepsilon)$. 
Write
\[
J_\varepsilon(u_n) = \frac12 \|\nabla u_n\|_2^2 - \frac{1}{1-r} \int_\Omega (u_n+\varepsilon)^{1-r} \,dx - \frac1p \|u_n\|_p^p.
\]
Since $u_n \to u_\varepsilon$ strongly in $L^p(\Omega)$, we have $\|u_n\|_p^p \to \|u_\varepsilon\|_p^p$. 
For the singular term, note that $(u_n+\varepsilon)^{1-r} \to (u_\varepsilon+\varepsilon)^{1-r}$ almost everywhere. 
To apply Vitali's convergence theorem, we verify uniform integrability. 
Because $\{u_n\}$ is bounded in $H_0^1(\Omega)$, it is also bounded in $L^{2^*}(\Omega)$. 
Since $0<r<1$, we have $0<1-r<1$. The sequence $\{(u_n+\varepsilon)^{1-r}\}$ is bounded in $L^{\frac{2^*}{1-r}}(\Omega)$ because
\[
\int_\Omega (u_n+\varepsilon)^{(1-r)\cdot \frac{2^*}{1-r}} \,dx = \int_\Omega (u_n+\varepsilon)^{2^*} \,dx \le C,
\]
where $C$ depends on the bound of $\{u_n\}$ in $L^{2^*}(\Omega)$ and $\varepsilon$. 
Since $\frac{2^*}{1-r} > 1$, boundedness in $L^{\frac{2^*}{1-r}}(\Omega)$ implies uniform integrability. 
Hence, $\{(u_n+\varepsilon)^{1-r}\}$ is uniformly integrable, and Vitali's theorem gives
\[
\int_\Omega (u_n+\varepsilon)^{1-r} \,dx \to \int_\Omega (u_\varepsilon+\varepsilon)^{1-r} \,dx.
\]
For the gradient term, we have weak lower semicontinuity:
\[
\frac12 \|\nabla u_\varepsilon\|_2^2 \le \liminf_{n\to\infty} \frac12 \|\nabla u_n\|_2^2.
\]
Therefore,
\[
c_\varepsilon \le J_\varepsilon(u_\varepsilon) \le \liminf_{n\to\infty} J_\varepsilon(u_n) = c_\varepsilon,
\]
so $J_\varepsilon(u_\varepsilon) = c_\varepsilon$. 
Moreover, this equality forces $\|\nabla u_\varepsilon\|_2^2 = \lim_{n\to\infty} \|\nabla u_n\|_2^2$, which together with the weak convergence implies $u_n \to u_\varepsilon$ strongly in $H_0^1(\Omega)$.

Since $c_\varepsilon < 0$ and by Lemma~\ref{lem:geometry} we have $J_\varepsilon(u) > 0$ for any $u \in S_\rho$ with $\|\nabla u\|_2 = \tau$, it must be that $\|\nabla u_\varepsilon\|_2 < \tau$.
Thus, $u_\varepsilon$ is a minimizer of $J_\varepsilon$ on $K$. Because $\|\nabla u_\varepsilon\|_2 < \tau$, the gradient constraint is not active at $u_\varepsilon$; consequently, $u_\varepsilon$ is in fact a local minimizer of $J_\varepsilon$ on the manifold $S_\rho$.

Now, we prove that $u_\varepsilon > 0$ in $\Omega$. Since $u_\varepsilon$ is a local minimizer on $S_\rho$, the Lagrange multiplier rule applies. The constraint functional $G(u) = \int_\Omega u^2 \,dx$ has derivative $G'(u_\varepsilon)v = 2\int_\Omega u_\varepsilon v \,dx$, which is a nonzero linear functional on $H_0^1(\Omega)$ because $u_\varepsilon \not\equiv 0$. Hence, there exists $\widetilde{\lambda}_\varepsilon \in \R$ such that $J_\varepsilon'(u_\varepsilon) = \widetilde{\lambda}_\varepsilon G'(u_\varepsilon)$. 
Computing the derivatives, for any $v \in H_0^1(\Omega)$,
\[
\int_\Omega \nabla u_\varepsilon \nabla v \,dx - \int_\Omega (u_\varepsilon+\varepsilon)^{-r} v \,dx - \int_\Omega u_\varepsilon^{p-1} v \,dx = 2\widetilde{\lambda}_\varepsilon \int_\Omega u_\varepsilon v \,dx.
\]
Denoting $\lambda_\varepsilon = 2\widetilde{\lambda}_\varepsilon$, we obtain the weak formulation
\[
\int_\Omega \nabla u_\varepsilon \nabla v \,dx =  \int_\Omega (u_\varepsilon+\varepsilon)^{-r} v \,dx + \int_\Omega u_\varepsilon^{p-1} v \,dx + \lambda_\varepsilon \int_\Omega u_\varepsilon v \,dx.
\]
Thus $u_\varepsilon$ satisfies in the weak sense
\[
-\Delta u_\varepsilon = (u_\varepsilon+\varepsilon)^{-r} + u_\varepsilon^{p-1} + \lambda_\varepsilon u_\varepsilon \quad \text{in } \Omega, \qquad u_\varepsilon = 0 \quad \text{on } \partial\Omega.
\]
Since $\varepsilon > 0$, the right-hand side is strictly positive wherever $u_\varepsilon = 0$. 
By elliptic regularity, because the right-hand side is bounded (as $u_\varepsilon \ge 0$ and $\varepsilon > 0$), we have $u_\varepsilon \in C^{1,\alpha}(\Omega)$ for some $\alpha \in (0,1)$. 
Suppose, by contradiction, that there exists $x_0 \in \Omega$ such that $u_\varepsilon(x_0) = 0$. 
Since $u_\varepsilon \ge 0$, $x_0$ is a local minimum point, so $\nabla u_\varepsilon(x_0) = 0$ and $\Delta u_\varepsilon(x_0) \ge 0$. 
But then the left-hand side satisfies $-\Delta u_\varepsilon(x_0) \le 0$, while the right-hand side at $x_0$ equals $\varepsilon^{-r} > 0$, a contradiction. 
Therefore, $u_\varepsilon$ cannot vanish in $\Omega$, and since $u_\varepsilon \ge 0$ and $u_\varepsilon \not\equiv 0$, we conclude that $u_\varepsilon > 0$ in $\Omega$.
\end{proof}

So far, we have constructed, for each $\varepsilon>0$, a positive solution $u_\varepsilon$ of the regularized problem. In the next section we will derive uniform estimates for these solutions and pass to the limit $\varepsilon\to0^+$.

\section{Uniform estimates and passage to the limit}\label{sec:estimates}

In this section, we pass to the limit $\varepsilon\to0^+$ in the regularized problems to recover a solution to the original singular problem.
Let $\varepsilon_k \to 0^+$ and denote $(\lambda_k,u_k) := (\lambda_{\varepsilon_k},u_{\varepsilon_k})$.
 Each pair satisfies the regularized equation
\begin{equation}\label{eq:approx_problem}
-\Delta u_k + \lambda_k u_k = (u_k+\varepsilon_k)^{-r} + u_k^{p-1} \quad \text{in } \Omega, \qquad u_k = 0 \text{ on } \partial\Omega,
\end{equation}
with $u_k \in S_{\rho}$, $u_k > 0$ in $\Omega$, and $\|\nabla u_k\|_2 < \tau$.  The passage to the limit requires uniform estimates on the sequences $\{\lambda_k\}$ and $\{u_k\}$.

\subsection{Uniform bounds}

We begin with a lower bound for the singular term that is independent of $\varepsilon$.

\begin{lemma}\label{lem:lowerbound}
There exists a constant $\beta > 0$ independent of $\varepsilon$ such that
\[
\int_\Omega u_\varepsilon^{1-r} \,dx \ge \beta, \qquad \forall \varepsilon \in (0,1].
\]
\end{lemma}

\begin{proof}
We argue by contradiction. If no such $\beta$ exists, then there exists a sequence $\{\varepsilon_k\} \subset (0,1]$ with $\varepsilon_k \to 0$ such that
\[
\lim_{k\to\infty} \int_\Omega u_{\varepsilon_k}^{1-r} \,dx = 0.
\]
Set $u_k := u_{\varepsilon_k}$.
Since the sequence $\{u_k\}$ is bounded in $H_0^1(\Omega)$. Passing to a subsequence (still denoted by $\{u_k\}$), there exists $u \in H_0^1(\Omega)$ such that $u_k \rightharpoonup u$ weakly in $H_0^1(\Omega)$. By the compact embedding $H_0^1(\Omega) \hookrightarrow L^1(\Omega)$, we may also assume $u_k \to u$ strongly in $L^1(\Omega)$ and almost everywhere.

We now show that $\int_\Omega u_k^{1-r} \,dx \to \int_\Omega u^{1-r} \,dx$. To apply Vitali's convergence theorem, we verify uniform integrability of the sequence $\{u_k^{1-r}\}$. Because $\{u_k\}$ is bounded in $H_0^1(\Omega)$, it is also bounded in $L^{2^*}(\Omega)$; let $C>0$ be such that $\|u_k\|_{2^*} \le C$ for all $k$. Using H\"older's inequality with exponents $\frac{2^*}{1-r}$ and $\frac{2^*}{2^*-(1-r)}$, we obtain for any measurable set $E \subset \Omega$,
\[
\int_E u_k^{1-r} \,dx \le |E|^{\frac{2^*-(1-r)}{2^*}} \left( \int_E u_k^{2^*} \,dx \right)^{\frac{1-r}{2^*}} \le |E|^{\frac{2^*-(1-r)}{2^*}} C^{1-r}.
\]
Since $|E|^{\frac{2^*-(1-r)}{2^*}} \to 0$ as $|E| \to 0$, the sequence $\{u_k^{1-r}\}$ is uniformly integrable. Together with the pointwise convergence $u_k^{1-r} \to u^{1-r}$ a.e., Vitali's theorem yields
\[
\lim_{k\to\infty} \int_\Omega u_k^{1-r} \,dx = \int_\Omega u^{1-r} \,dx.
\]

By our assumption, the left-hand side equals $0$, so $\int_\Omega u^{1-r} \,dx = 0$. Since $u^{1-r} \ge 0$, we deduce that $u^{1-r} = 0$ almost everywhere, which means $u = 0$ almost everywhere.

Now, using again the compact embedding $H_0^1(\Omega) \hookrightarrow L^2(\Omega)$, we have (for a further subsequence if necessary) $u_k \to u$ strongly in $L^2(\Omega)$. Because $u = 0$ a.e., it follows that $\|u_k\|_2 \to 0$. However, each $u_k$ belongs to $S_{\rho}$, so $\|u_k\|_2^2 = \rho > 0$ for all $k$. This gives
\[
0 < \rho = \lim_{k\to\infty} \|u_k\|_2^2 = 0,
\]
a contradiction.

Consequently, our initial assumption must be false, and we conclude that there exists  $\beta > 0$ such that
\[
\int_\Omega u_\varepsilon^{1-r} \,dx \ge \beta \quad \text{for all } \varepsilon \in (0,1].
\]
This completes the proof.
\end{proof}

%This lower bound guarantees that the solutions $u_\varepsilon$ do not vanish in the $L^{1-r}$ sense as $\varepsilon \to 0$, which will be particularly important in the subsequent blow-up analysis.

% 引理7：Lagrange乘子的有界性
Denote by $\lambda_1 > 0$ the first eigenvalue of $-\Delta$ on $\Omega$ with Dirichlet boundary condition, and let $\varphi_1 \in H_0^1(\Omega)$ be the corresponding positive eigenfunction normalized by $\|\varphi_1\|_{L^2(\Omega)} = 1$. Thus $\varphi_1$ satisfies
\[
-\Delta \varphi_1 = \lambda_1 \varphi_1 \quad \text{in } \Omega, \qquad \varphi_1 = 0 \text{ on } \partial\Omega.
\]

The following lemma provides the necessary uniform estimates for the Lagrange multipliers.

\begin{lemma}\label{lem:boundlambda_k}
Let $\varepsilon_k \to 0^+$ and let $(\lambda_k, u_k)$ be the solutions obtained in Lemma~\ref{lem:minimizer} corresponding to $\varepsilon = \varepsilon_k$. Then there exists a constant $M > 0$ such that
\[
-\lambda_1 < \lambda_k \le M \qquad \text{for all } k.\]
\end{lemma}

\begin{proof}
We establish the upper and lower bounds separately.

\textbf{Lower bound $\lambda_k > -\lambda_1$.}
Each pair $(\lambda_k, u_k)$ satisfies the weak formulation
\begin{equation}\label{eq:weak_form}
\int_\Omega \nabla u_k  \nabla \psi \,dx + \lambda_k \int_\Omega u_k \psi \,dx
= \int_\Omega (u_k+\varepsilon_k)^{-r} \psi \,dx + \int_\Omega u_k^{p-1} \psi \,dx
\end{equation}
for all $\psi \in H_0^1(\Omega)$. Choose $\psi = \varphi_1$, the first eigenfunction of $-\Delta$ on $H_0^1(\Omega)$, normalized with $\varphi_1 > 0$ and $\|\varphi_1\|_2 = 1$. Using the identity
\[
\int_\Omega \nabla u_k  \nabla \varphi_1 \,dx = \lambda_1 \int_\Omega u_k \varphi_1 \,dx,
\]
then \eqref{eq:weak_form} becomes
\[
(\lambda_1 + \lambda_k) \int_\Omega u_k \varphi_1 \,dx
= \int_\Omega (u_k+\varepsilon_k)^{-r} \varphi_1 \,dx + \int_\Omega u_k^{p-1} \varphi_1 \,dx.
\]
Since $u_k > 0$, $\varphi_1 > 0$, and $\varepsilon_k > 0$, the right-hand side is strictly positive. Moreover, $\int_\Omega u_k \varphi_1 \,dx > 0$. Consequently,
$\lambda_1 + \lambda_k > 0$ implies that $\lambda_k > -\lambda_1$.

\textbf{Upper bound $\lambda_k \le M$.} 
Testing \eqref{eq:weak_form} with $\psi = u_k$ yields
\[
\int_\Omega |\nabla u_k|^2 \,dx + \lambda_k \int_\Omega u_k^2 \,dx
= \int_\Omega (u_k+\varepsilon_k)^{-r} u_k \,dx + \int_\Omega u_k^p \,dx.
\]
Since $\|u_k\|_2^2 = \rho$ and $\int_\Omega |\nabla u_k|^2 \,dx \ge 0$, we obtain
\begin{equation}\label{eq:lambda_estimate}
\lambda_k \rho \le \int_\Omega (u_k+\varepsilon_k)^{-r} u_k \,dx + \int_\Omega u_k^p \,dx.
\end{equation}

We now estimate the two integrals on the right-hand side by constants independent of $k$.

\emph{Power term.} By the Gagliardo-Nirenberg inequality, there exists a constant $C = C(N,p,\Omega) > 0$ such that
\[
\|u_k\|_p \le C \|\nabla u_k\|_2^a \|u_k\|_2^{1-a} \le C \tau^a \rho^{(1-a)/2},
\]
where $a = N\left(\frac12 - \frac1p\right) \in (0,1)$. Hence,
\begin{equation}\label{eq:power_bound}
\int_\Omega u_k^p \,dx \le C^p \tau^{ap} \rho^{(1-a)p/2} =: C_p.
\end{equation}

\emph{Singular term.} Observe that
\[
0 \le (u_k+\varepsilon_k)^{-r} u_k \le (u_k+\varepsilon_k)^{1-r}.
\]
Using the elementary inequality $(a+b)^{1-r} \le 2^{1-r}(a^{1-r}+b^{1-r})$ for $a,b \ge 0$, we obtain
\[
\int_\Omega (u_k+\varepsilon_k)^{1-r} \,dx \le 2^{1-r} \int_\Omega u_k^{1-r} \,dx + 2^{1-r} \varepsilon_k^{1-r} |\Omega|.
\]
As $\varepsilon_k \to 0$, the term $\varepsilon_k^{1-r} |\Omega|$ tends to zero, but for a uniform bound we simply note that $\varepsilon_k^{1-r} \le 1$ for $\varepsilon_k \le 1$ (which holds for large $k$). Thus,
\[
\int_\Omega (u_k+\varepsilon_k)^{1-r} \,dx \le 2^{1-r} \int_\Omega u_k^{1-r} \,dx + 2^{1-r} |\Omega|.
\]
To bound $\int_\Omega u_k^{1-r} \,dx$, apply H\"older's inequality with exponents $\frac{2}{1-r}$ and $\frac{2}{1+r}$ (note that $0<r<1$):
\[
\int_\Omega u_k^{1-r} \,dx \le |\Omega|^{\frac{1+r}{2}} \left( \int_\Omega u_k^2 \,dx \right)^{\frac{1-r}{2}} = |\Omega|^{\frac{1+r}{2}} \rho^{\frac{1-r}{2}}.
\]
Therefore,
\begin{equation}\label{eq:singular_bound}
 \int_\Omega (u_k+\varepsilon_k)^{-r} u_k \,dx \le  \int_\Omega (u_k+\varepsilon_k)^{1-r} \,dx
\le \left( 2^{1-r} |\Omega| + 2^{1-r} |\Omega|^{\frac{1+r}{2}} \rho^{\frac{1-r}{2}} \right) =: C_s.
\end{equation}
Inserting the bounds \eqref{eq:power_bound} and \eqref{eq:singular_bound} into \eqref{eq:lambda_estimate} gives
\[
\lambda_k \rho \le C_s + C_p.
\]
Since $\rho > 0$, we may set $M := (C_s + C_p)/\rho$, which yields $\lambda_k \le M$ for all $k$.

Combining the two parts, we conclude that $-\lambda_1 < \lambda_k \le M$ for every $k$.
\end{proof}

% 引理8：一致有界性（爆破分析）

To control the singular term $u^{-r}$ in the limit, we also need a uniform $L^\infty$ estimate for $\{u_k\}$. This is achieved by a blow-up argument.

\begin{lemma}\label{lem:uniform_bound}
Let $\varepsilon_k \to 0^+$ and let $(\lambda_k, u_k)$ be the solutions obtained in Lemma~\ref{lem:minimizer} corresponding to $\varepsilon = \varepsilon_k$. Then there exists a constant $C>0$ such that $\|u_k\|_{L^\infty(\Omega)} \le C$ for all $k$.
\end{lemma}

\begin{proof}
We argue by contradiction. Suppose that $\|u_k\|_{L^\infty(\Omega)} \to \infty$ along a subsequence (still denoted by $\{u_k\}$).
Set $M_k := \|u_k\|_{L^\infty(\Omega)}$. Following the regularity, we can choose $x_k \in \Omega$ such that $u_k(x_k) = M_k$.
Let $d_k := \operatorname{dist}(x_k, \partial\Omega)$.
By passing to a further subsequence, we may assume either
\begin{enumerate}
    \item $d_k \to d > 0$ (interior case), or
    \item $d_k \to 0$ (boundary case).
\end{enumerate}
We treat the interior case in detail; the boundary case is similar and will be sketched afterwards.

\noindent\textbf{Case 1: Interior case $d_k \to d > 0$.}

Without loss of generality, we assume $x_k \to x_0 \in \Omega$ and that there exists $R_0 > 0$ such that $\overline{B_{2R_0}(x_0)} \subset \Omega$.
For sufficiently large $k$, we have $x_k \in B_{R_0}(x_0)$ and $B_{R_0}(x_k) \subset B_{2R_0}(x_0) \subset \Omega$.

\noindent\textbf{Step 1: Scaling.}
Define
\[
\delta_k := M_k^{-\frac{p-2}{2}}, \qquad
v_k(y) := \frac{1}{M_k} u_k\bigl(x_k + \delta_k y\bigr), \quad y \in \Omega_k := \bigl\{ y\in\mathbb{R}^N: x_k + \delta_k y \in \Omega \bigr\}.
\]
Then $0 \le v_k \le 1$, $v_k(0)=1$, and a direct calculation shows that $v_k$ satisfies the rescaled equation
\begin{equation}\label{eq:scaled_eq}
-\Delta v_k = M_k^{-(p-1+r)} \bigl(v_k + \varepsilon_k M_k^{-1}\bigr)^{-r} + v_k^{\,p-1} - \lambda_k M_k^{-(p-2)} v_k \quad \text{in } \Omega_k.
\end{equation}
Indeed, recall that $u_k$ satisfies
\[
-\Delta u_k + \lambda_k u_k = (u_k+\varepsilon_k)^{-r} + u_k^{\,p-1}.
\]
Substituting $u_k(x) = M_k v_k\bigl((x-x_k)/\delta_k\bigr)$ and using the chain rule, we obtain
\[
-\frac{1}{\delta_k^2} \Delta v_k + \lambda_k M_k v_k = \bigl(M_k v_k + \varepsilon_k\bigr)^{-r} + M_k^{p-1} v_k^{\,p-1}.
\]
Multiplying by $\delta_k^2 = M_k^{-(p-2)}$ and rearranging yields \eqref{eq:scaled_eq}.

\noindent\textbf{Step 2: Uniform positivity on compact sets.}
Since the first two terms on the right-hand side of \eqref{eq:scaled_eq} are nonnegative, we have
\[
-\Delta v_k \ge -\lambda_k M_k^{-(p-2)} v_k.
\]
By Lemma~\ref{lem:boundlambda_k}, the sequence $\{\lambda_k\}$ is bounded; hence there exists a constant $\Lambda > 0$ such that $|\lambda_k| \le \Lambda$ for all $k$.
Because $M_k \to \infty$, we have $M_k^{-(p-2)} \to 0$, so there exists $k_0$ such that for all $k \ge k_0$, $|\lambda_k| M_k^{-(p-2)} \le 1$.
Consequently,
\[
-\Delta v_k \ge - v_k \qquad \text{in } \Omega_k \text{ for } k \ge k_0.
\]
Fix any $R > 0$. For $k$ large enough, we have $B_{2R}(0) \subset \Omega_k$.
Apply the Harnack inequality to the inequality $-\Delta v_k + v_k \ge 0$ in $B_{2R}(0)$.
There exists a constant $C_H = C_H(N,R) > 0$, independent of $k$, such that
\[
\sup_{B_R(0)} v_k \le C_H \inf_{B_R(0)} v_k.
\]
Since $v_k(0)=1$, it follows that $\inf_{B_R(0)} v_k \ge C_H^{-1}$.
Therefore, for every compact set $W \subset \mathbb{R}^N$, there exists a constant $c_W > 0$ such that
$$
v_k(y) \ge c_W \qquad \text{for all } y \in W \text{ and all sufficiently large } k.
$$

\noindent\textbf{Step 3: Convergence to a limit.}
The uniform lower bound on compact sets implies that the singular term in \eqref{eq:scaled_eq} satisfies, for $y \in W$,
\[
0 \le M_k^{-(p-1+r)} \bigl(v_k(y) + \varepsilon_k M_k^{-1}\bigr)^{-r} \le M_k^{-(p-1+r)} c_W^{-r} \to 0.
\]
Moreover, $|\lambda_k| M_k^{-(p-2)} \to 0$.
Hence, the right-hand side of \eqref{eq:scaled_eq} converges locally uniformly to $v^{p-1}$.
Standard elliptic $L^p$ estimates together with the Sobolev embedding yield that $\{v_k\}$ is bounded in $C^{1,\alpha}_{\text{loc}}(\mathbb{R}^N)$ for some $\alpha \in (0,1)$.
After extracting a subsequence, we have $v_k \to v$ in $C^1_{\text{loc}}(\mathbb{R}^N)$, and the limit function $v$ satisfies
\begin{equation}\label{eq:limit_eq}
-\Delta v = v^{\,p-1}, \quad v \ge 0, \quad v(0)=1.
\end{equation}
By the strong maximum principle, $v > 0$ in $\mathbb{R}^N$.

\noindent\textbf{Step 4: Energy contradiction.}
For any fixed $R > 0$, the convergence $v_k \to v$ in $C^1(\overline{B_R(0)})$ implies
\[
\int_{B_R(0)} v_k^p \,dy \to \int_{B_R(0)} v^p \,dy > 0.
\]
Rescaling back to $u_k$, we obtain
\[
\int_{B(x_k, R\delta_k)} u_k^p \,dx = M_k^p \delta_k^N \int_{B_R(0)} v_k^p \,dy = M_k^{p - \frac{N(p-2)}{2}} \int_{B_R(0)} v_k^p \,dy.
\]
Since we assume $p > 2 + \frac{4}{N}$ (which is exactly the condition $p - \frac{N(p-2)}{2} > 0$), we have $M_k^{p - \frac{N(p-2)}{2}} \to \infty$.
Consequently,
\[
\int_\Omega u_k^p \,dx \ge \int_{B(x_k, R\delta_k)} u_k^p \,dx \to \infty.
\]
On the other hand, by the Gagliardo-Nirenberg inequality and the facts that $\|\nabla u_k\|_2 \le \tau$ and $\|u_k\|_2 = \sqrt{\rho}$, there exists a constant $C_0 = C_0(N,p,\Omega) > 0$ such that
\[
\|u_k\|_p \le C_0 \|\nabla u_k\|_2^{a} \|u_k\|_2^{1-a} \le C_0 \tau^{a} \rho^{\frac{1-a}{2}}, \qquad a = N\left( \frac12 - \frac1p \right).
\]
Thus, $\|u_k\|_p^p$ is uniformly bounded, contradicting the above blow-up.
Hence, the interior case cannot occur.

\noindent\textbf{Case 2: Boundary case $d_k \to 0$.}

In this case, we may assume (after a translation and rotation) that $x_k \to 0 \in \partial\Omega$ and that near $0$, the boundary $\partial\Omega$ is flat, say $\Omega \cap B_{2r}(0) = \{ x \in B_{2r}(0) : x_N > 0 \}$ for some $r>0$.
For simplicity, we assume $\Omega$ is a half-ball near $0$; the general case can be reduced to this by a local flattening of the boundary, which does not affect the blow-up argument.

Define $M_k$ and $\delta_k$ as before.
Now choose $y_k \in \partial\Omega$ such that $|x_k - y_k| = d_k$, and set $\tilde{x}_k := y_k$ (the projection of $x_k$ onto $\partial\Omega$).
After translating $\tilde{x}_k$ to the origin and rotating, we may assume that in the new coordinates, $\Omega \supset \{ x : x_N > 0 \}$ and $x_k = (0,\dots,0,d_k)$ with $d_k \to 0$.

We scale as before, but now we define
\[
v_k(y) := \frac{1}{M_k} u_k\bigl( \delta_k y', \delta_k y_N + d_k \bigr), \qquad y = (y', y_N) \in \Omega_k,
\]
where $\Omega_k := \{ y : (\delta_k y', \delta_k y_N + d_k) \in \Omega \}$.
For large $k$, $\Omega_k$ contains the set $\{ y : y_N > -d_k/\delta_k \}$.
Notice that $d_k/\delta_k$ may tend to $0$, $\infty$, or a finite positive number.
However, by a standard argument, 
one can show that, after passing to a subsequence, $d_k/\delta_k \to \infty$ leads to a contradiction because the limiting problem would be on the whole space, but the boundary condition is lost.
If $d_k/\delta_k \to c \ge 0$, then the limiting domain is the half-space $\{ y_N > -c \}$, and the limit function $v$ satisfies
\[
-\Delta v = v^{\,p-1} \ \text{in } \{ y_N > -c \}, \qquad v = 0 \ \text{on } \{ y_N = -c \},
\]
with $v(0)=1$ (since $x_k = (0,d_k)$ corresponds to $y^{(k)} = (0, d_k/\delta_k) \to (0,c)$).
By the maximum principle, $v > 0$ in the interior.
However, for $p < 2^*$, there is no positive solution of $-\Delta v = v^{p-1}$ in a half-space with Dirichlet boundary condition.
Thus, we again obtain a contradiction.

Since both cases lead to contradictions, our initial assumption that $\|u_k\|_{L^\infty(\Omega)} \to \infty$ is false.
Therefore, there exists a constant $C > 0$ such that $\|u_k\|_{L^\infty(\Omega)} \le C$ for all $k$.
\end{proof}

\subsection{Convergence to a solution of the original problem}

With the uniform bounds established in Lemmas~\ref{lem:boundlambda_k} and \ref{lem:uniform_bound}, we can now pass to the limit $\varepsilon_k \to 0$ and obtain a solution of the original problem \eqref{eq:main}. The next lemma extracts a convergent subsequence.

\begin{lemma}\label{lem:convergence}
Under the assumptions of Lemma~\ref{lem:boundlambda_k}, there exist a subsequence (still denoted by $\{(\lambda_k, u_k)\}$), a function $u \in H_0^1(\Omega) \cap L^\infty(\Omega)$, and a number $\lambda \in \R$ such that the following hold:
\begin{enumerate}
    \item[(1)] $u_k \rightharpoonup u$ weakly in $H_0^1(\Omega)$,
    \item[(2)] $u_k \to u$ strongly in $L^q(\Omega)$ for every $q \in [1,\infty)$,
    \item[(3)] $\lambda_k \to \lambda$,
    \item[(4)] $u \in S_{\rho}$, i.e., $u \ge 0$ and $\displaystyle\int_\Omega u^2\,dx = \rho$,
    \item[(5)] $\displaystyle\int_\Omega u^{1-r}\,dx \ge \beta > 0$.
\end{enumerate}
\end{lemma}

\begin{proof}
We prove each statement in order.

\textbf{(1) Weak convergence in $H_0^1(\Omega)$.}
Since $u_k \in K$, we have $\|\nabla u_k\|_2 \le \tau$ and $\|u_k\|_2 = \sqrt{\rho}$. Consequently,
\[
\|u_k\|_{H_0^1}^2 = \|\nabla u_k\|_2^2 + \|u_k\|_2^2 \le \tau^2 + \rho,
\]
so $\{u_k\}$ is bounded in $H_0^1(\Omega)$. By the reflexivity of $H_0^1(\Omega)$, there exists a subsequence (still denoted by $\{u_k\}$) and a function $u \in H_0^1(\Omega)$ such that
\[
u_k \rightharpoonup u \quad \text{weakly in } H_0^1(\Omega).
\]

\textbf{(2) Strong convergence in $L^q(\Omega)$ for all $q \in [1,\infty)$.}
By the Rellich--Kondrachov compactness theorem, the embedding $H_0^1(\Omega) \hookrightarrow L^q(\Omega)$ is compact for every $q \in [1, 2^*)$. Hence, from the weak convergence $u_k \rightharpoonup u$ in $H_0^1(\Omega)$, we obtain strong convergence in $L^q(\Omega)$ for every $q \in [1, 2^*)$. In particular, $u_k \to u$ in $L^2(\Omega)$.

From Lemma~\ref{lem:uniform_bound}, there exists a constant $C_\infty > 0$ such that $\|u_k\|_{L^\infty(\Omega)} \le C_\infty$ for all $k$. Since $u_k \to u$ in $L^2(\Omega)$, we deduce that $u \in L^\infty(\Omega)$ with $\|u\|_{L^\infty} \le C_\infty$ (indeed, any $L^2$-limit of a sequence uniformly bounded in $L^\infty$ must belong to $L^\infty$ with the same bound).

Now fix an arbitrary $q \in [1, \infty)$. Choose $q_0 \in [1, 2^*)$ such that $q_0 \le q$. By the interpolation inequality, we have
\[
\|u_k - u\|_{L^q} \le \|u_k - u\|_{L^{q_0}}^{\theta} \|u_k - u\|_{L^\infty}^{1-\theta},
\]
where $\theta = q_0/q \in (0,1]$. Since $\|u_k - u\|_{L^\infty} \le 2C_\infty$ and $\|u_k - u\|_{L^{q_0}} \to 0$ as $k \to \infty$, it follows that $\|u_k - u\|_{L^q} \to 0$. Thus $u_k \to u$ strongly in $L^q(\Omega)$ for every $q \in [1,\infty)$.

\textbf{(3) Convergence of $\lambda_k$.}
By Lemma~\ref{lem:boundlambda_k}, the sequence $\{\lambda_k\}$ is bounded in $\R$. Therefore, extracting a further subsequence if necessary, we have $\lambda_k \to \lambda$ for some $\lambda \in \R$.

\textbf{(4) $u \in S_{\rho}$.}
Because $u_k \to u$ strongly in $L^2(\Omega)$ and $\int_\Omega u_k^2 \,dx = \rho$ for all $k$, we obtain
\[
\int_\Omega u^2 \,dx = \lim_{k \to \infty} \int_\Omega u_k^2 \,dx = \rho.
\]
Moreover, since $u_k \ge 0$ almost everywhere, the pointwise limit (which exists after passing to a subsequence, and the limit is identified with $u$ due to $L^2$-convergence) satisfies $u \ge 0$ almost everywhere. Hence, $u \in S_{\rho}$.

\textbf{(5) Lower bound for the singular term.}
From Lemma~\ref{lem:lowerbound}, we have $\int_\Omega u_k^{1-r} \,dx \ge \beta > 0$ for all $k$. Note that $0 < 1-r < 1$ because $0<r<1$. We claim that $u_k \to u$ in $L^{1-r}(\Omega)$. Indeed, using H\"older's inequality with exponent $\frac{1}{1-r}$ and its conjugate $\frac{1}{r}$, we obtain
\[
\int_\Omega |u_k - u|^{1-r} \,dx \le |\Omega|^{r} \left( \int_\Omega |u_k - u| \,dx \right)^{1-r}.
\]
Since $u_k \to u$ in $L^1(\Omega)$ (due to the strong convergence in $L^2(\Omega)$ and the boundedness of $\Omega$), the right-hand side tends to zero. Therefore,
\[
\int_\Omega u^{1-r} \,dx = \lim_{k \to \infty} \int_\Omega u_k^{1-r} \,dx \ge \beta > 0.
\]

This completes the proof of all five statements.
\end{proof}

Now we verify that the limit $(\lambda,u)$ satisfies the weak formulation of the original problem.

\begin{lemma}\label{lem:limit_solution}
Let $(\lambda, u)$ be the limit obtained in Lemma~\ref{lem:convergence}. Then for every $\varphi \in C_c^\infty(\Omega)$,
\begin{equation}\label{eq:weak_limit}
\int_\Omega \nabla u \nabla \varphi \,dx + \lambda \int_\Omega u \varphi \,dx = \int_\Omega u^{-r} \varphi \,dx + \int_\Omega u^{p-1} \varphi \,dx.
\end{equation}
Thus, $(\lambda,u)$ is a weak solution of \eqref{eq:main}.
\end{lemma}

\begin{proof}
Fix $\varphi \in C_c^\infty(\Omega)$. From Lemma~\ref{lem:minimizer}, each $(\lambda_k, u_k)$ satisfies the weak formulation
\begin{equation}\label{eq:weak_form_k}
\int_\Omega \nabla u_k \nabla \varphi \,dx + \lambda_k \int_\Omega u_k \varphi \,dx =  \int_\Omega (u_k+\varepsilon_k)^{-r} \varphi \,dx + \int_\Omega u_k^{\,p-1} \varphi \,dx.
\end{equation}
We will pass to the limit $k \to \infty$ in each term of \eqref{eq:weak_form_k}. By Lemma~\ref{lem:convergence}, we have along a subsequence (still denoted by the same indices) that $u_k \rightharpoonup u$ in $H_0^1(\Omega)$, $u_k \to u$ in $L^q(\Omega)$ for all $q \in [1,\infty)$, and $\lambda_k \to \lambda$.

We first consider the gradient term. Since $u_k \rightharpoonup u$ in $H_0^1(\Omega)$, for every $\varphi \in C_c^\infty(\Omega)$ we have
\begin{equation}\label{eq:limit_grad}
\int_\Omega \nabla u_k \nabla \varphi \,dx \longrightarrow \int_\Omega \nabla u \nabla \varphi \,dx.
\end{equation}
For the linear term, because $u_k \to u$ in $L^2(\Omega)$ and $\lambda_k \to \lambda$, we obtain
\begin{equation}\label{eq:limit_linear}
\lambda_k \int_\Omega u_k \varphi \,dx \longrightarrow \lambda \int_\Omega u \varphi \,dx.
\end{equation}
Next, we treat the power term. Since $p < 2^*$ and $u_k \to u$ in $L^p(\Omega)$, the sequence $\{u_k^{p-1}\}$ converges to $u^{p-1}$ in $L^{p/(p-1)}(\Omega)$. As $\varphi$ is bounded and has compact support, H\"older's inequality yields
\begin{equation}\label{eq:limit_power}
\int_\Omega u_k^{\,p-1} \varphi \,dx \longrightarrow \int_\Omega u^{p-1} \varphi \,dx.
\end{equation}
It remains to handle the singular term. We need to prove that
\begin{equation}\label{eq:limit_singular}
\int_\Omega (u_k+\varepsilon_k)^{-r} \varphi \,dx \longrightarrow \int_\Omega u^{-r} \varphi \,dx.
\end{equation}
Let $U := \operatorname{supp} \varphi \subset \subset \Omega$. We first show that the sequence $\{u_k\}$ is uniformly bounded away from zero on $U$.

Since $\lambda_k \le M$ for all $k$, each $u_k$ satisfies
\[
-\Delta u_k + M u_k\ge -\Delta u_k + \lambda_k u_k = (u_k+\varepsilon_k)^{-r} + u_k^{\,p-1} \ge 0 \quad \text{in } \Omega,
\]
the function $u_k$ is a nonnegative solution of the inequality $-\Delta u_k + M u_k \ge 0$ in $\Omega$. Consequently, the Harnack inequality for the operator $-\Delta + M$ applies with a constant that depends only on $M$, $N$, and $\Omega$, but not on $k$. Hence, there exists a constant $C_H > 0$, independent of $k$, such that for any ball $B_{2R}(x) \subset \Omega$,
\begin{equation}\label{eq:harnack}
   \sup_{B_R(x)} u_k \le C_H \inf_{B_R(x)} u_k. 
\end{equation}
Cover the compact set $U$ by finitely many balls $B_R(x_i)$ with $B_{2R}(x_i) \subset \Omega$. 

%From Lemma~\ref{lem:uniform_bound}, the sequence $\{u_k\}$ is uniformly bounded in $L^\infty(\Omega)$; let $C_\infty > 0$ be such that $\|u_k\|_{L^\infty(\Omega)} \le C_\infty$ for all $k$. 

%Then on each ball $B_R(x_i)$,
%\textcolor{red}{
%\[
%\inf_{B_R(x_i)} u_k \ge \frac{1}{C_H} \sup_{B_R(x_i)} u_k \ge \frac{C_\infty}%{C_H}.
%\]
%}
We claim that there exists a constant $c_0 > 0$, independent of $k$, such that for each $k$, at least one ball $B_R(x_{i_k})$ satisfies
\begin{equation}\label{eq:claim}
  \sup_{B_R(x_{i_k})} u_k \ge c_0.  
\end{equation}
Assume, by contradiction, that $c_0$ does not exist. 
Then for every integer $\ell \ge 1$, there exist infinitely many indices $k$ such that
\[
\sup_{B_R(x_i)} u_k < \frac{1}{\ell} \quad \text{for all } i = 1,\dots,m.
\]
We can select a strictly increasing sequence $\{k_\ell\}_{\ell=1}^\infty$ such that for each $\ell$,
\begin{equation}\label{eq:c0}
   \sup_{B_R(x_i)} u_{k_\ell} < \frac{1}{\ell} \quad \text{for all } i = 1,\dots,m. 
\end{equation}
Since $\{B_R(x_i)\}_{i=1}^m$ cover $U$, \eqref{eq:c0} implies that for every $x \in U$,
\[
u_{k_\ell}(x) < \frac{1}{\ell}.
\]
Hence, $u_{k_\ell} \to 0$ uniformly in $U$.
Now take a test function $\varphi \in C_c^\infty(\Omega)$ with $\operatorname{supp} \varphi = U$ and $\varphi > 0$ in $\operatorname{int}(U)$. Multiplying the equation
\[
-\Delta u_k + \lambda_k u_k = (u_k + \varepsilon_k)^{-r} + u_k^{p-1}
\]
by $\varphi$ and integrating by parts, we obtain
\begin{equation}\label{eq:lem3.5}
  \int_\Omega \nabla u_k \nabla \varphi \,dx + \lambda_k \int_\Omega u_k \varphi \,dx = \int_U (u_k + \varepsilon_k)^{-r} \varphi \,dx + \int_U u_k^{p-1} \varphi \,dx.  
\end{equation}
From Lemma~\ref{lem:uniform_bound}, the sequence $\{u_k\}$ is uniformly bounded in $L^\infty(\Omega)$, then by standard elliptic estimates, $\{u_k\}$ is uniformly bounded in $H_0^1(\Omega)$. Moreover, $\{\lambda_k\}$ is bounded. Thus, the left-hand side of \eqref{eq:lem3.5} remains bounded as $k \to \infty$. However, for the subsequence $\{u_{k_\ell}\}$, the right-hand side satisfies
\[
\int_U (u_{k_\ell} + \varepsilon_{k_\ell})^{-r} \varphi \,dx \ge \int_U u_{k_\ell}^{-r} \varphi \,dx > \ell^{\,r} \int_U \varphi \,dx \to \infty \quad \text{as } \ell \to \infty,
\]
while
\[
\int_U u_{k_\ell}^{p-1} \varphi \,dx \le \| u_{k_\ell} \|_{L^\infty(U)}^{p-1} \int_U \varphi \,dx \to 0.
\]
This contradicts the boundedness of the left-hand side. Therefore, the claim holds.

By \eqref{eq:harnack} and \eqref{eq:claim}, there exists a ball where $u_k$ has a uniform positive lower bound $\frac{c_0}{C_H}$. Since $U$ is covered by finitely many balls and the Harnack constant $C_H$ is independent of $k$, we can propagate this lower bound to the whole set $U$ via a finite chain of intersecting balls. Consequently, there exists a constant $c > 0$ such that
\[
u_k(x) \ge c \qquad \text{for all } x \in U \text{ and all } k.\]

Now, on $U$ we have
\[
0 \le (u_k+\varepsilon_k)^{-r} \le u_k^{-r} \le c^{-r},
\]
because $u_k \ge c$ and $\varepsilon_k \ge 0$. Moreover, since $u_k \to u$ in $L^1(\Omega)$, there exists a subsequence (still denoted by $u_k$) such that $u_k \to u$ almost everywhere. Also, $\varepsilon_k \to 0$. It follows that $(u_k+\varepsilon_k)^{-r} \to u^{-r}$ almost everywhere in $U$. The constant function $c^{-r}$ is integrable over $U$ (as $U$ has finite measure). Therefore, by the dominated convergence theorem,
\[
\int_U (u_k+\varepsilon_k)^{-r} \varphi \,dx \longrightarrow \int_U u^{-r} \varphi \,dx.
\]
Since $\varphi$ vanishes outside $U$, this yields \eqref{eq:limit_singular}.

Finally, combining the limits \eqref{eq:limit_grad}--\eqref{eq:limit_singular} and returning to equation \eqref{eq:weak_form_k}, we obtain
\[
\int_\Omega \nabla u \nabla \varphi \,dx + \lambda \int_\Omega u \varphi \,dx = \int_\Omega u^{-r} \varphi \,dx + \int_\Omega u^{p-1} \varphi \,dx,
\]
which is exactly the weak formulation \eqref{eq:weak_limit}. Hence, $(\lambda,u)$ is a weak solution of the original problem \eqref{eq:main}.
\end{proof}

Finally, we establish the strict positivity of the limiting solution.

\begin{lemma}\label{lem:positivity}
Let $(\lambda,u)$ be the weak solution obtained in Lemma~\ref{lem:limit_solution}. Then $u > 0$ in $\Omega$.
\end{lemma}

\begin{proof}
From Lemma~\ref{lem:convergence} we have $u \ge 0$ in $\Omega$ and $u \not\equiv 0$ because $\|u\|_2 = \sqrt{\rho} > 0$.
The weak formulation \eqref{eq:weak_limit} corresponds to the pointwise equation
\[
-\Delta u + \lambda u =  u^{-r} + u^{p-1} \quad \text{in } \Omega,
\]
which holds classically after applying elliptic regularity: since $u \in H_0^1(\Omega) \cap L^\infty(\Omega)$, we have $u \in C^{1,\alpha}_{\mathrm{loc}}(\Omega)$ for some $\alpha \in (0,1)$.
The right-hand side is nonnegative because $0<r<1$ and $u \ge 0$. Consequently, $u$ satisfies
\[
-\Delta u + \lambda u \ge 0 \quad \text{in } \Omega.
\]
If $u$ were zero at some interior point $x_0 \in \Omega$, then $x_0$ would be a minimum point of $u$.
By the strong maximum principle, a nonnegative solution of $-\Delta u + \lambda u \ge 0$ cannot attain an interior zero minimum unless it is identically zero.
Since $u \not\equiv 0$, we conclude that $u$ cannot vanish anywhere in $\Omega$; hence $u > 0$ in $\Omega$.
\end{proof}

\subsection{Proof of Theorem~\ref{thm:main}}

With all the pieces in place, we can now complete the proof of our main theorem.

\begin{proof}[Proof of Theorem~\ref{thm:main}]
The proof proceeds by constructing a solution as the limit of solutions to the regularized problems.

\textbf{Step 1: Choice of parameters.}
Let $\rho_0>0$ be the constant provided by Lemma~\ref{lem:geometry}. By Lemma~\ref{lem:negative_energy}, after possibly reducing $\rho_0$ further, we have $c_\varepsilon<0$ for every $\varepsilon\in(0,1]$. Fix an arbitrary $\rho\in(0,\rho_0)$.

\textbf{Step 2: Solutions of the regularized problems.}
For each $\varepsilon>0$, Lemma~\ref{lem:minimizer} yields a pair $(\lambda_\varepsilon,u_\varepsilon)$ satisfying
\[
-\Delta u_\varepsilon + \lambda_\varepsilon u_\varepsilon =(u_\varepsilon+\varepsilon)^{-r}+u_\varepsilon^{\,p-1}\quad\text{in }\Omega,
\]
with $u_\varepsilon>0$ in $\Omega$, $u_\varepsilon\in S_{\rho}$, and $\|\nabla u_\varepsilon\|_2<\tau$.

\textbf{Step 3: Passage to the limit.}
Take a sequence $\varepsilon_k\to0^+$ and set $(\lambda_k,u_k):=(\lambda_{\varepsilon_k},u_{\varepsilon_k})$.
Lemma~\ref{lem:boundlambda_k} shows that $\{\lambda_k\}$ is bounded in $\R$, and Lemma~\ref{lem:uniform_bound} gives a uniform bound for $\{u_k\}$ in $L^\infty(\Omega)$.
Hence, by Lemma~\ref{lem:convergence}, there exists a subsequence (still denoted by $\{(\lambda_k,u_k)\}$), a function $u\in H_0^1(\Omega)\cap L^\infty(\Omega)$, and a number $\lambda\in\R$ such that
\[
u_k\rightharpoonup u\ \text{weakly in }H_0^1(\Omega),\quad
u_k\to u\ \text{strongly in }L^q(\Omega)\ (\forall\,q\in[1,\infty)),\quad
\lambda_k\to\lambda,
\]
and $u\in S_{\rho}$ (so $u\ge0$ and $\int_\Omega u^2\,dx=\rho$). Moreover, $\int_\Omega u^{1-r}\,dx\ge\beta>0$.

\textbf{Step 4: Verification that the limit satisfies the original equation.}
According to Lemma~\ref{lem:limit_solution}, the limit pair $(\lambda,u)$ satisfies
\[
\int_\Omega\nabla u \nabla\varphi\,dx+\lambda\int_\Omega u\varphi\,dx
   =\int_\Omega u^{-r}\varphi\,dx+\int_\Omega u^{p-1}\varphi\,dx
\qquad\forall\varphi\in C_c^\infty(\Omega),
\]
which is exactly the definition of a weak solution to \eqref{eq:main}.

\textbf{Step 5: Positivity of the limit.}
Finally, Lemma~\ref{lem:positivity} guaranties that $u>0$ everywhere in $\Omega$.

Thus, for every $\rho\in(0,\rho_0)$ we have found a pair $(\lambda,u)\in \R\times(H_0^1(\Omega)\cap L^\infty(\Omega))$ with $u>0$ in $\Omega$ that solves \eqref{eq:main}. This completes the proof of Theorem~\ref{thm:main}.
\end{proof}

\section{Conclusion}\label{sec:conclusion}

In this paper, we have proved the existence of positive solutions to the singular elliptic problem \eqref{eq:main} with a prescribed $L^2$-norm. Our main result, Theorem~\ref{thm:main}, guaranties the existence of a solution when the prescribed mass $\rho>0$ is sufficiently small. The proof relies on a variational method combined with a regularization argument, and a crucial step is the derivation of uniform estimates via a blow-up analysis.

%Several natural questions remain open for future investigation:
%\begin{itemize}
   % \item \textbf{Multiplicity:} Does problem \eqref{eq:main} admit multiple positive solutions for certain ranges of $\rho$?
    %\item \textbf{Uniqueness:} Under what conditions is the positive solution unique?
    %\item \textbf{Asymptotic behavior:} How do solutions behave as $\rho$ approaches zero or increases toward a critical threshold?
   % \item \textbf{Generalizations:} Can similar results be obtained for more general singular terms or for problems involving fractional Laplacian?
%\end{itemize}

We expect that the approach and results presented here will serve as a useful foundation for the further study of singular elliptic equations with mass constraints.

%\section*{Acknowledgments}

\end{document}